\documentclass[12pt,reqno]{article}

\usepackage[usenames]{color}
\usepackage[colorlinks=true,
linkcolor=webgreen, filecolor=webbrown,
citecolor=webgreen]{hyperref}

\definecolor{webgreen}{rgb}{0,.5,0}
\definecolor{webbrown}{rgb}{.6,0,0}

\usepackage{amssymb}
\usepackage{graphicx}
\usepackage{amscd}
\usepackage{lscape}
\usepackage{tikz}
\usepackage{tikz-cd}
\usepackage{pgfplots}

\usetikzlibrary{matrix}
\usetikzlibrary{fit,shapes}
\usetikzlibrary{positioning, calc}
\tikzset{circle node/.style = {circle,inner sep=1pt,draw, fill=white},
        X node/.style = {fill=white, inner sep=1pt},
        dot node/.style = {circle, draw, inner sep=5pt}
        }
\usepackage{tkz-fct}

\usepackage{amsthm}
\newtheorem{theorem}{Theorem}

\newtheorem{proposition}[theorem]{Proposition}
\newtheorem{corollary}[theorem]{Corollary}

\theoremstyle{definition}

\newtheorem{example}[theorem]{Example}

\usepackage{float}

\usepackage{graphics,amsmath}
\usepackage{amsfonts}
\usepackage{latexsym}
\usepackage{epsf}

\newcommand{\seqnum}[1]{\href{http://oeis.org/#1}{\underline{#1}}}

\begin{document}

\begin{center}
\vskip 1cm{\LARGE\bf Constant coefficient Laurent biorthogonal polynomials, Riordan arrays and moment sequences} \vskip 1cm \large
Paul Barry\\
School of Science\\
Waterford Institute of Technology\\
Ireland\\
\href{mailto:pbarry@wit.ie}{\tt pbarry@wit.ie}
\end{center}
\vskip .2 in

\begin{abstract} We study properties of constant coefficient Laurent biorthogonal polynomials using Riordan arrays. We give details of related orthogonal polynomials, and we explore relationships between the moments of these orthogonal polynomials, the moments of the defining Laurent biorthogonal polynomials, and the expansions of $T$-fractions. Closed form expressions are given for the polynomials and their moments. \end{abstract}

\section{Introduction}

Let $b_{n+1}$ and $c_n$ for $n \in \mathbb{N}$ be arbitrary nonzero constants. The monic \emph{Laurent biorthogonal polynomials} (LBPs) $P_n(x)$ \cite{PB, Kam, Zhedanov}, defined by the sequences $b_n$ and $c_n$, is the sequence of polynomials determined by the recurrence
$$P_n(x)=(x-c_{n-1})P_{n-1}(x)-b_{n-1}xP_{n-2}(x), \quad \text{for}\quad n\ge 1,$$ with
$P_0(x)=1, P_1(x)=x-c_0$. The LBP $P_n(x)$ is a monic polynomial in $x$ of exact degree $n$ of which the constant term does not vanish. This note will be concerned with the constant coefficient case, that is, we assume that $c_n=c$ and $b_n=b$. 

We let $P_n(x)=\sum_{k=0}^n a_{n,k}x^k$, and we call the lower triangular matrix $(a_{n,k})_{0 \le n,k \le \infty}$ the coefficient array of the family of LBPs $\{P_n(x)\}$. In the case that the defining sequences are constant ($b_n=b$ and $c_n=c$) we have the following result \cite{PB}.
\begin{proposition} The coefficient array of the family of LBPs $\{P_n(x)\}$ defined by
$$P_n(x)=(x-c)P_{n-1}(x)-bxP_{n-2}(x), \quad \text{for}\quad n\ge 1,$$ with
$P_0(x)=1, P_1(x)=x-c$ is given by the Riordan array \cite{SGWW}
$$\left(\frac{1}{1+ct}, \frac{x(1-bt)}{1+ct}\right).$$
\end{proposition} The Riordan array $\left(\frac{1}{1+ct}, \frac{x(1-bt)}{1+ct}\right)$ has been called a generalized Delannoy matrix \cite{Yang}.
We recall that the Riordan array $(g(t), f(t))$, where 
$$g(t)=g_0 + g_1 t + g_2 t^2+ \ldots,$$ and 
$$f(t)=f_1 t + f_2 t^2+ f_3 t^2+ \ldots,$$ is the matrix with general $(n,k)$-th term given by 
$$a_{n,k}=[x^n]g(t)f(t)^k.$$ Here, $[x^n]$ is the functional which extracts the coefficient of $x^n$ \cite{MC}. 

From the above proposition, we have the following results.
\begin{corollary} We have the following generating function for the family $\{P_n(x)\}$ in the case of constant coefficients.
$$\sum_{n=0} P_n(x)t^n=\frac{1}{1+ct+xt(bt-1)}.$$
\end{corollary}
\begin{proof}
The result follows since the bivariate generating function of the Riordan array $\left(\frac{1}{1+ct}, \frac{y(1-bt)}{1+ct}\right)$ is given by
$$\frac{\frac{1}{1+ct}}{1-x\frac{t(1-bt)}{1+ct}}=\frac{1}{1+ct+xt(bt-1)}.$$
\end{proof}
\begin{corollary}
We have
$$P_n(x)=\sum_{k=0}^n \left(\sum_{j=0}^k \binom{k}{j}\binom{n-j}{n-k-j}(-b)^j(-c)^{n-k-j}\right)x^k.$$
\end{corollary}
\begin{proof} We have $P_n(x)=\sum_{k=0}^n a_{n,k}x^k$, where $a_{n,k}$ is given by
$$a_{n,k}=[t^n] \frac{1}{1+c t} \left(\frac{x(1-bt)}{1+ct}\right)^k.$$
Expanding this by the method of coefficients \cite{MC} leads to the desired expression.
\end{proof}

It is appropriate to call the inverse of the matrix $(a_{n,k})$ the \emph{moment matrix} of the LBPs $\{P_n(x)\}$, and to designate the elements $\mu_n$ of its first column as the \emph{moments}. We then have
\begin{proposition}
The moments $\mu_n$ of the LBPs $\{P_n(x)\}$ have generating function
\begin{align*}\mu(x)&=\frac{c+2b-c^2t-c \sqrt{1-2(2b+c)t+c^2t^2}}{2b}\\
&=1+\frac{ct}{1-ct}C\left(\frac{bt}{(1-ct)^2}\right),\end{align*}
where
$$C(t)=\frac{1-\sqrt{1-4t}}{2t}$$ is the generating function of the Catalan numbers \seqnum{A000108}.
\end{proposition}
\begin{proof} In effect, the inverse of the Riordan array $\left(\frac{1}{1+ct}, \frac{t(1-bt)}{1+ct}\right)$ is given by
$$\left(\frac{1}{1+ct}, \frac{t(1-bt)}{1+ct}\right)^{-1}=\left(1+\frac{ct}{1-ct}C\left(\frac{bt}{(1-ct)^2}\right), \frac{t}{1-ct}C\left(\frac{bt}{(1-ct)^2}\right)\right).$$ 
\end{proof}

\begin{corollary} We have
$$\mu_n= \sum_{k=0}^n \binom{2n-k-1}{2n-2k}C_{n-k} b^{n-k}c^k=0^n+c \sum_{k=0}^{n-1}\binom{n+k-1}{2k}C_k c^{n-k-1}b^k,$$ where
$C_n=\frac{1}{n+1}\binom{2n}{n}$ is the $n$-th Catalan number.
\end{corollary}
The moments begin
$$1, c, c(b + c), c(b + c)(2b + c), c(b + c)(5b^2 + 5bc + c^2),\ldots.$$
\begin{corollary}
We have
$$\mu(t)=
\cfrac{1}
{1-\cfrac{ct}
{1-\cfrac{bt}
{1-\cfrac{(b+c)t}
{1-\cfrac{bt}
{1-\cfrac{(b+c)t}
{1-\cdots}}}}}}.$$
\end{corollary}
\begin{proof}
We solve for $u=u(t)$ where 
$$u=\frac{1}{1-\frac{bt}{1-(b+c)t}}.$$ 
Then we have $\mu(x)=\frac{1}{1-ctu}$. 
\end{proof}
\begin{corollary} We have
$$\mu(t)=
\cfrac{1}{1-ct-
\cfrac{bct^2}{1-(2b+c)t-
\cfrac{b(b+c)t^2}{1-(2b+c)t-
\cfrac{b(b+c)t^2}{1-(2b+c)t-\cdots}}}}.$$
\end{corollary}
\begin{corollary} The Hankel transform $h_n=|\mu_{i+j}|_{0 \le i,j \le n}$ of the moments $\mu_n$ is given by
$$h_n= (bc)^n (b(b+c))^{\binom{n}{2}}.$$
\end{corollary}
\begin{proof} This follows from Heilermann's formula \cite{Kratt} since the coefficients of $t^2$ are
$$bc, b(b+c), b(b+c), b(b+c),\ldots.$$
\end{proof}
\begin{corollary} The moments $\mu_n$ of the LBPs $\{P_n(x)\}$ defined by 
$$P_n(x)=(x-c)P_{n-1}-bxP_{n-2}(x)$$ with $P_0(x)=1, P_1(x)=x-c$, are also the moments for the family of orthogonal polynomials $Q_n(x)$ whose coefficient array is given by the Riordan array 
$$\left(\frac{(1+bt)^2}{1+(2b+c)t+b(b+c)t^2}, \frac{t}{1+(2b+c)t+b(b+c)t^2}\right).$$
\end{corollary}
\begin{proof}
Calculation shows that the inverse of the above Riordan array is 
$$\left(\mu(x), \frac{1-(2b+c)t-\sqrt{1-2(2b+c)t+c^2t^2}}{2b(b+c)t}\right).$$
\end{proof}
The orthogonal polynomials $Q_n(x)$ satisfy
$$Q_n(x)=(x-(2b+c))Q_{n-1}(x)-b(b+c)Q_{n-2}(x),$$ with 
$$Q_0(x)=1, Q_1(x)=x-c, Q_2(x)=x^2-2x(b+c)+c(b+c).$$ 
For an exposition of the links between Riordan arrays and orthogonal polynomials, see \cite{Classical}.

We finish this section by noting that we can use Lagrange inversion \cite{LI} to determine the elements of the inverse coefficient matrix $(a_{n,k})^{-1}=\left(\frac{1}{1+ct}, \frac{t(1-bt}{1+ct}\right)^{-1}$. We obtain the following result.
\begin{proposition} The $(n,k)$-th element of the inverse of the coefficient matrix $(a_{n,k})$ is given by 
$$\frac{k}{n}\sum_{j=0}^n \binom{n}{j}\binom{2n-k-j-1}{n-k-j}c^jb^{n-k-j}+\frac{c(k+1)}{n}\sum_{j=0}^n \binom{n}{j}\binom{2n-k-j-2}{n-k-j-1}c^jb^{n-k-j-1}.$$ 
\end{proposition}
Setting $k=0$ in this expression (which gives us the first column), we obtain that $\mu_0=1$ and 
$$\mu_n=\frac{c}{n}\sum_{j=0}^n \binom{n}{j}\binom{2n-j-2}{n-j-1}c^jb^{n-j-1}$$ for $n \ge 1$. 

\section{T-fractions and Toeplitz determinants}
We begin this section by considering the $T$-fraction  \cite{Jones}
$$\tilde{\mu}(t)=\cfrac{1}{1-ct-
\cfrac{bt}{1-ct-
\cfrac{bt}{1-ct-\cdots}}}.$$ 
Solving the equation 
$$u = \frac{1}{1-ct-btu}$$ for $u=u(t)$, we see that 
$$\tilde{\mu}(t)=u(t)=\frac{1-ct-\sqrt{1-2(2b+c)t+c^2t^2}}{2bt}.$$ 
This expands to give the sequence $\tilde{\mu}_n$  that begins 
$$1, b + c, (b + c)(2b + c), (b + c)(5b^2 + 5bc + c^2),\ldots.$$
We have 
$$\mu(t)=1+t \tilde{\mu}(t).$$ 
\begin{proposition} The generating function $\tilde{\mu}(t)$ is the generating function of the moments of the family of orthogonal polynomials $\tilde{Q}_n(x)$ whose coefficient array is given by the Riordan array 
$$\left(\frac{1+bt}{1+(2b+c)t+b(b+c)t^2},\frac{t}{1+(2b+c)t+b(b+c)t^2}\right).$$ 
\end{proposition}
\begin{proof} By the theory of Riordan arrays, the generating function of the first column of the inverse matrix is equal to $\tilde{\mu}(t)$. 
\end{proof}
The orthogonal polynomials $\tilde{Q}_n(x)$ satisfy the recurrence
$$\tilde{Q}_n(x)=(x-(2b+c))\tilde{Q}_{n-1}(x)-b(b+c)\tilde{Q}_{n-2}(x),$$ with 
$$\tilde{Q}_0(x)=1, \tilde{Q}_1(x)=x-(b+c).$$
\begin{proposition} 
We have 
$$\tilde{\mu}_n=\sum_{k=0}^n \binom{n+k}{2k}c^{n-k}b^k C_k.$$ 
\end{proposition}. 
\begin{proof} The matrix with general $(n,k)$-th term $\binom{n+k}{2k}c^{n-k}$ is the Riordan array 
$\left(\frac{1}{1-ct}, \frac{t}{(1-ct)^2}\right)$. Applying this Riordan array to the generating function 
$$\cfrac{1}{1-
\cfrac{bt}{1-
\cfrac{bt}{1-\cdots}}}$$ of the sequence $b^n C_n$, we obtain the generating function 
$$\tilde{\mu}(t)=\cfrac{1}{1-ct-
\cfrac{bt}{1-ct-
\cfrac{bt}{1-ct-\cdots}}}$$ of $\tilde{\mu}_n$. 
\end{proof}
\begin{corollary} We have
$$\mu_n=0^n+\sum_{k=0}^{n-1} \binom{n+k-1}{2k}c^{n-k-1}b^k C_k.$$
\end{corollary}
The theory of LBPs is tied to that of $T$-fractions and Toeplitz determinants. In order to define the relevant Toeplitz matrices and determinants in our case, we must extend the moments $(\mu_n)_{n \ge 0}$ to a bi-infinite sequence $(\mu_n)_{-\infty \le n \le \infty}$. We do this as follows. We extend $\mu_n$ to $n < 0$ by setting 
$$\mu_n = \frac{\mu_{1-n}}{c^{1-2n}},\quad\text{for}\quad n<0.$$
We let 
$$t_n=|\mu_{-j+k}|_{j,k=0\cdots n}$$ and 
$$t'_n=|\mu_{1-j+k}|_{j,k=0 \cdots n}.$$ 
The theory of LBPs now gives us the following results.
\begin{proposition}
We have 
$$b=-\frac{t_{n-1}t'_{n+1}}{t_n t'_n}.$$
$$c=\frac{t_n t'_{n+1}}{t_{n+1} t'_n}.$$ 
The Toeplitz transform $t_n$ of the moments $\mu_n$ of the LBPs $\{P_n(x)\}$ is given by
$$t_n = \left(-\frac{b}{c}\right)^{\binom{n+1}{2}}.$$ 
The polynomials $P_n(x)$ are given by 
$$P_n(x)=\frac{1}{t_n} \left|
\begin{array}{ccccc}
 \mu_0 & \mu_1 & \cdots & \mu_{n-1} & \mu_n \\
 \mu_{-1} & \mu_0 & \cdots & \mu_{n-2} & \mu_{n-1} \\
 \vdots & \vdots & \ddots & \vdots & \vdots \\
 \mu_{-n+1} & \mu_{-n+2} & \cdots & \mu_0 & \mu_1 \\
 1 & x & \cdots & x^{n-1} & x^n \\
\end{array}
\right|.$$
\end{proposition}

\section{Examples}

\begin{example} In this example, we consider the case where $b=1$. By solving the equations 
$$u = \frac{1}{1+c t +t u}$$ and 
$$v=\frac{1}{1+\frac{(c+1)t}{1+tv}}$$ and comparing $u$ and $v$, we see that the two continued fractions
$$u(t)=\cfrac{1}{1+ct+
\cfrac{t}{1+ct+
\cfrac{t}{1+ct+\ldots}}}$$ and 
$$v(t)=\cfrac{1}{1+
\cfrac{(c+1)t}{1+
\cfrac{t}{1+
\cfrac{(c+1)t}{1+
\cfrac{t}{1+\cdots}}}}}$$ are equal. 
Their common expansion $\tilde{\mu}_n$ begins 
$$1, c + 1, (c + 1)(c + 2), (c + 1)(c^2 + 5c + 5), (c + 1)(c^3 + 9c^2 + 21c + 14), \ldots.$$ 
As a sequence of polynomials in $c$, the corresponding coefficient array begins
$$\left(
\begin{array}{ccccccc}
 1 & 0 & 0 & 0 & 0 & 0 & 0 \\
 1 & 1 & 0 & 0 & 0 & 0 & 0 \\
 2 & 3 & 1 & 0 & 0 & 0 & 0 \\
 5 & 10 & 6 & 1 & 0 & 0 & 0 \\
 14 & 35 & 30 & 10 & 1 & 0 & 0 \\
 42 & 126 & 140 & 70 & 15 & 1 & 0 \\
 132 & 462 & 630 & 420 & 140 & 21 & 1 \\
\end{array}
\right).$$ 
This is triangle \seqnum{A060693}, with general term $\binom{2n-k}{k}C_{n-k}$, which counts the number of Schroeder paths from $(0,0)$ to $(2n,0)$ with $k$ peaks. We have in this case
$$\tilde{\mu}_n = \sum_{k=0}^n \binom{2n-k}{k}C_{n-k}c^k.$$ 
The above theory tells us that the sequence $\tilde{\mu}_n$ is the moment sequence for the orthogonal polynomials with coefficient array given by the Riordan array 
$$\left(\frac{1+t}{1+(c+2)t+(c+1)t^2}, \frac{t}{1+(c+2)t+(c+1)t^2}\right).$$ 
The corresponding LBP moment sequence $\mu_n$, which begins 
$$1, c, c(c + 1), c(c + 1)(c + 2), c(c + 1)(c^2 + 5c + 5), c(c + 1)(c^3 + 9c^2 + 21c + 14), \ldots,$$ is given by the first column of the inverse of the Riordan array 
$$\left(\frac{1}{1+ct}, \frac{t(1-t)}{1+ct}\right).$$ It is also given by the first column of the inverse of 
$$\left(\frac{(1+t)^2}{1+(c+2)t+(c+1)t^2}, \frac{t}{1+(c+2)t+(c+1)t^2}\right).$$ 
The link to Schroeder numbers is evident by taking $c=1$. In that case, 
we have 
$$\tilde{\mu}_n = \sum_{k=0}^n \binom{2n-k}{k}C_{n-k}=\sum_{k=0}^n \binom{n+k}{2k}C_k=S_n,$$ the $n$-th large Schroeder number \seqnum{A006318}.
In this case the numbers $\tilde{\mu}_n$ begin 
$$1, 2, 6, 22, 90, 394, 1806, 8558, 41586,\ldots,$$ while the LBP sequence $\mu_n$ begins 
$$1,1, 2, 6, 22, 90, 394, 1806, 8558, 41586,\ldots.$$
By varying the parameter $c$, we find interpretations in terms of Schroeder paths where the level steps can have $c$ colors. 
\end{example}

\begin{example} Given a Riordan array $A$, the matrix $P_A=A^{-1}\bar{A}$ is called its production matrix, where $\bar{A}$ denotes the matrix $A$ with its top row removed. For the LBP array $\left(\frac{1}{1+ct}, \frac{t(1-bt)}{1+ct}\right)$, its production matrix begins 
$$\left(
\begin{array}{cccccc}
 c & 1 & 0 & 0 & 0 & 0 \\
 b c & b+c & 1 & 0 & 0 & 0 \\
 b^2 c & b (b+c) & b+c & 1 & 0 & 0 \\
 b^3 c & b^2 (b+c) & b (b+c) & b+c & 1 & 0 \\
 b^4 c & b^3 (b+c) & b^2 (b+c) & b (b+c) & b+c & 1 \\
 b^5 c & b^4 (b+c) & b^3 (b+c) & b^2 (b+c) & b (b+c) & b+c \\
\end{array}
\right).$$ 
This displays a structural property of Riordan arrays: after the first column, all columns have the same elements, apart from the descending zeros.

We now take an example with non-constant $b_n$. Thus we let $b_n$ be the sequence that begins 
$$1,2,1,2,1,2,1,2,1,\ldots,$$ while we take $c_n=c=1$. 
Then the moment matrix corresponding to the LBPs given by 
$$P_n(x)=(x-1)P_{n-1}(x)- b_{n-1}xP_{n-2}(x),$$ with $P_0(x)=1, P_1(x)=x-1$, begins 
$$\left(
\begin{array}{cccccccc}
 1 & 0 & 0 & 0 & 0 & 0 & 0 & 0 \\
 1 & 1 & 0 & 0 & 0 & 0 & 0 & 0 \\
 3 & 4 & 1 & 0 & 0 & 0 & 0 & 0 \\
 13 & 18 & 6 & 1 & 0 & 0 & 0 & 0 \\
 65 & 91 & 34 & 9 & 1 & 0 & 0 & 0 \\
 355 & 500 & 199 & 64 & 11 & 1 & 0 & 0 \\
 2061 & 2914 & 1206 & 430 & 90 & 14 & 1 & 0 \\
 12501 & 17721 & 7526 & 2856 & 670 & 135 & 16 & 1 \\
\end{array}
\right).$$ 
This array has a production matrix which begins
$$\left(
\begin{array}{ccccccc}
 1 & 1 & 0 & 0 & 0 & 0 & 0 \\
 2 & 3 & 1 & 0 & 0 & 0 & 0 \\
 2 & 3 & 2 & 1 & 0 & 0 & 0 \\
 4 & 6 & 4 & 3 & 1 & 0 & 0 \\
 4 & 6 & 4 & 3 & 2 & 1 & 0 \\
 8 & 12 & 8 & 6 & 4 & 3 & 1 \\
 8 & 12 & 8 & 6 & 4 & 3 & 2 \\
\end{array}
\right).$$ 
This illustrates that the moment array is not a Riordan array; nevertheless, the production matrix does reflect the periodicity in the defining $b_n$ parameters.

The sequence \seqnum{A155867}, which begins 
$$1,3,13,65,355,2061,\ldots$$ is given by 
$$\sum_{k=0}^n \binom{n+k}{2k} S_k.$$
\end{example}

\begin{example} In this example, we consider the moments in the case where $b=c-1$. Thus we look at the first column of the matrix $\left(\frac{1}{1+ct}, \frac{t(1-(c-1)t)}{1+ct}\right)^{-1}$. 
The moments begin 
$$1, c, c(2c - 1), c(2c - 1)(3c - 2), c(2c - 1)(11c^2 - 15c + 5), c(2c - 1)(45c^3 - 93c^2 + 63c - 14),\ldots.$$
This sequence of polynomials in $c$ has a coefficient array that begins 
$$\left(
\begin{array}{cccccc}
 1 & 0 & 0 & 0 & 0 & 0 \\
 0 & 1 & 0 & 0 & 0 & 0 \\
 0 & -1 & 2 & 0 & 0 & 0 \\
 0 & 2 & -7 & 6 & 0 & 0 \\
 0 & -5 & 25 & -41 & 22 & 0 \\
 0 & 14 & -91 & 219 & -231 & 90 \\
\end{array}
\right).$$ 
The second column is composed of the alternating sign Catalan numbers, while the diagonal is given by the augmented large Schroeder numbers. The general $(n,k)$-th element of this matrix is given by 
$$(-1)^{n-k}\sum_{j=0}^n \binom{n+j-1}{2j}\binom{j}{n-k}C_j,$$ so that the moments in this case are give by
$$\mu_n=\sum_{k=0}^n (-1)^{n-k}\sum_{j=0}^n \binom{n+j-1}{2j}\binom{j}{n-k}C_j c^k.$$
The row sums of the matrix above are all equal to $1$. In this case, the Riordan array becomes $\left(\frac{1}{1+t}, \frac{t}{1+t}\right)^{-1}=\left(\frac{1}{1-t}, \frac{t}{1-t}\right)$, which is the binomial matrix with an all $1$s first column.

The unsigned matrix 
$$\left(
\begin{array}{cccccc}
 1 & 0 & 0 & 0 & 0 & 0 \\
 0 & 1 & 0 & 0 & 0 & 0 \\
 0 & 1 & 2 & 0 & 0 & 0 \\
 0 & 2 & 7 & 6 & 0 & 0 \\
 0 & 5 & 25 & 41 & 22 & 0 \\
 0 & 14 & 91 & 219 & 231 & 90 \\
\end{array}
\right)$$ 
corresponds to the case $b=c+1$. Its row sums begin 
$$1, 1, 3, 15, 93, 645,\ldots.$$
The sequence $0,1,3,15,93,\ldots$ is the expansion of the reversion of $\frac{x(1-2x)}{1+x}$, \seqnum{A103210}.
\end{example}
\begin{example} When $b=c$, we are dealing with the matrix $\left(\frac{1}{1+cx}, \frac{x(1-cx)}{1+cx}\right)$, which begins 
$$\left(
\begin{array}{cccccc}
 1 & 0 & 0 & 0 & 0 & 0 \\
 -c & 1 & 0 & 0 & 0 & 0 \\
 c^2 & -3 c & 1 & 0 & 0 & 0 \\
 -c^3 & 5 c^2 & -5 c & 1 & 0 & 0 \\
 c^4 & -7 c^3 & 13 c^2 & -7 c & 1 & 0 \\
 -c^5 & 9 c^4 & -25 c^3 & 25 c^2 & -9 c & 1 \\
\end{array}
\right).$$ 
The coefficients are those of the signed Delannoy triangle \seqnum{A008288}. In this case, the inverse or moment matrix begins 
$$\left(
\begin{array}{cccccc}
 1 & 0 & 0 & 0 & 0 & 0 \\
 c & 1 & 0 & 0 & 0 & 0 \\
 2 c^2 & 3 c & 1 & 0 & 0 & 0 \\
 6 c^3 & 10 c^2 & 5 c & 1 & 0 & 0 \\
 22 c^4 & 38 c^3 & 22 c^2 & 7 c & 1 & 0 \\
 90 c^5 & 158 c^4 & 98 c^3 & 38 c^2 & 9 c & 1 \\
\end{array}
\right).$$
In this case, the moments are the scaled large Schroeder numbers $\mu_n=c^n S_n$. 
\end{example}
\section{Relations between the Riordan arrays}
We let 
$$L=\left(\frac{1}{1+c t}, \frac{t(1-bt)}{1+ct}\right)$$ be the coefficient matrix of the LBPs $\{P_n(x)\}$. We recall that we have 
$$P_n(x)=(x-c)P_{n-1}(x)-bxP_{n-1}(x),$$ with $P_0(x)=1, P_1(x)=x-c$. 
The LBP moments $\mu_n$ are then the elements of the first column of $L^{-1}$. 
We let 
$$O=\left(\frac{(1+bt)^2}{1+(2b+c)t+b(b+c)t^2}, \frac{t}{1+(2b+c)t+b(b+c)t^2}\right).$$ 
This is the coefficient array of the family of orthogonal polynomials $Q_n(x)$ which satisfies the recurrence 
$$Q_n(x)=(x-(2b+c))Q_{n-1}(x)-b(b+c)Q_{n-2}(x),$$ with 
$Q_0(x)=1, Q_1(x)=x-c, Q_2(x)=x^2-2x(b+c)+c(b+c)$. 
The LBP moments $\mu_n$ are given by the elements of the first column of $O^{-1}$.

We also let 
$$\tilde{O}=\left(\frac{1+bt}{1+(2b+c)t+b(b+c)t^2}, \frac{t}{1+(2b+c)t+b(b+c)t^2}\right).$$ This is the coefficient matrix of the family of orthogonal polynomials $\tilde{Q}_n(x)$ that satisfy the recurrence 
$$\tilde{Q}_n(x)=(x-(2b+c))\tilde{Q}_{n-1}(x)-b(b+c)\tilde{Q}_{n-2}(x),$$ with
$\tilde{Q}_0(x)=1, \tilde{Q}_1(x)=x-(b+c)$. 
The moments $\tilde{\mu}_n$ are given by the elements of the first column of $\tilde{O}^{-1}$. 
Finally, we let 
$$B(b)=\left(\frac{1}{1-bt}, \frac{t}{1-bt}\right)$$ be the generalized binomial matrix with general $(n,k)$-th element $\binom{n}{k}b^{n-k}$. We have $B(b)^{-1}=B(-b)$. 
\begin{proposition}
We have 
$$L=\left(1, \frac{t}{1-bt}\right)\cdot O.$$
\end{proposition}
\begin{proof}
We have $O=(g(t), f(t))$ where $g(t)=\frac{(1+t)^2}{1+(2b+c)t+b(b+c)t^2}$. 
The first element of the product above is then given by 
$$1.g\left(\frac{t}{1-bt}\right)=\frac{1}{1+ct}.$$ 
Similarly the second element is given by 
$$f\left(\frac{t}{1-bt}\right)=\frac{t(1-bt)}{1+ct}.$$
\end{proof}
\begin{corollary}
We have 
$$P(n,x)=\sum_{k=0}^n \binom{n-1}{n-k}b^{n-k} Q_k(x).$$
\end{corollary}
\begin{proof}
The general element of the Riordan array $\left(1, \frac{t}{1-bt}\right)$ is $\binom{n-1}{n-k}b^{n-k}$. 
\end{proof}
This provides a Riordan array interpretation of the links between the constant coefficient Laurent biorthogonal polynomials and the orthogonal polynomials defined by the coefficient matrix $O$.
\begin{proposition}
We have 
$$L=B(b)\cdot \tilde{O}.$$
\end{proposition}
\begin{corollary} We have 
$$P_n(x)=\sum_{k=0}^n \binom{n}{k}b^{n-k} \tilde{Q}_k(x).$$
\end{corollary}
We can thus say that the LBP polynomials are the $b$-th binomial transform of the orthogonal polynomials $\tilde{Q}_n(x)$. The relationship between the polynomials $Q_n(x)$ and $\tilde{Q})_n(x)$ is governed by the following result. 
\begin{proposition}
We have 
$$O=(1+bt,t)\cdot \tilde{O}.$$
\end{proposition}
We can find a relationship between the LBP matrix $\left(\frac{1}{1+ct}, \frac{t(1-bt)}{1+ct}\right)$ and the simpler orthogonal polynomial coefficient array $\left(\frac{1}{1+(2b+c)t+b(b+c)t^2}, \frac{t}{1+(2b+c)t+b(b+c)t^2}\right)$. This is the array of the family of orthogonal polynomials $\hat{Q}_n(x)$ that satisfy
$$\hat{Q}_n(x)=(x-(2b+c))\hat{Q}_{n-1}(x)-b(b+c)\hat{Q}_{n-2}(x)$$ with 
$\hat{Q}_0(x)=1, \hat{Q}_1(x)=x-(2b+c)$. 
\begin{proposition} 
We have 
$$L=\left(\frac{1}{(1-bt)^2}, \frac{t}{1-bt}\right)\cdot \left(\frac{1}{1+(2b+c)t+b(b+c)t^2}, \frac{t}{1+(2b+c)t+b(b+c)t^2}\right).$$
\end{proposition}
We conclude from this that 
$$P_n(x)=\sum_{k=0}^n \binom{n+1}{k+1}b^{n-k}\hat{Q}_k(x).$$

\section{Conclusions}
The structure of Laurent biorthogonal polynomials defined by constant coefficients are fully defined by the properties of the generalized Delannoy matrix given by the Riordan array $\left(\frac{1}{1+ct}, \frac{t(1-bt)}{1+ct}\right)$. By means of Riordan array analysis, we can show them to be the binomial transform of a related family of orthogonal polynomials. These constant coefficient Laurent biorthogonal polynomials can be defined by $T$-fractions which are related to colored Schroeder paths. In particular, the moments of these LBP polynomials count Schroeder paths with colored horizontal and down steps (where the colors are the same for each level).

\bigskip
\hrule
\bigskip
\noindent 2010 {\it Mathematics Subject Classification}:
Primary 42C05; Secondary
11B83, 11C20, 15B05, 15B36, 33C45.

\noindent \emph{Keywords:} Laurent biorthogonal polynomials, orthogonal polynomials, moments, Toeplitz determinant, Hankel determinant, Riordan array.

\bigskip
\hrule
\bigskip
\noindent (Concerned with sequences
\seqnum{A000108},
\seqnum{A006318}, 
\seqnum{A060693}, 
\seqnum{A103210}, and
\seqnum{A155867}.)

\end{document}